\documentclass{article}

\usepackage{graphicx}
\usepackage{amssymb, amsmath, amsthm, epsfig, url}

\usepackage{thmtools}
\usepackage{verbatim}
\usepackage{caption}
\usepackage{subcaption}
\declaretheorem[name=Example,qed={\lower-0.3ex\hbox{$\triangle$}}]{example}

\newtheorem{theorem}{Theorem}[section]
\newtheorem{lemma}[theorem]{Lemma}
\newtheorem{proposition}[theorem]{Proposition}
\newtheorem{definition}[theorem]{Definition}
\newtheorem{remark}{Remark}

\newcommand{\bfA}{\mathbf{A}}
\newcommand{\bfb}{\mathbf{b}}
\newcommand{\bfc}{\mathbf{c}}
\newcommand{\bft}{\mathbf{t}}
\newcommand{\bfm}{\mathbf{m}}

\newcommand{\bff}{\mathbf{f}}
\newcommand{\bfp}{\mathbf{p}}
\newcommand{\bfu}{\mathbf{u}}
\newcommand{\bfv}{\mathbf{v}}
\newcommand{\x}{\mathbf{x}}

\newcommand{\z}{\mathbf{z}}

\newcommand{\rr}{{\mathbb R}}

\newcommand{\CC}{{\mathbb C}}
\newcommand{\VV}{{\mathbb V}}
\newcommand{\aug}{{\rm aug}}

\begin{document}

\title{The Method of Gauss-Newton to Compute Power Series Solutions
of Polynomial Homotopies\thanks{This material is based upon work 
supported by the National Science Foundation under Grant No. 1440534.
Date: 25 October 2017.}}

\renewcommand{\thefootnote}{\fnsymbol{footnote}} 
\renewcommand{\thefootnote}{\arabic{footnote}} 

\author{Nathan Bliss \and Jan Verschelde}

\date{University of Illinois at Chicago \\
Department of Mathematics, Statistics, and Computer Science \\
851 S. Morgan Street (m/c 249), Chicago, IL 60607-7045, USA \\
{\tt {\{nbliss2,janv\}@uic.edu}}}

\maketitle

\begin{abstract}
We consider the extension of the method of Gauss-Newton from
complex floating-point arithmetic to the field of truncated
power series with complex floating-point coefficients.
With linearization we formulate a linear system where the
coefficient matrix is a series with matrix coefficients, and
provide a characterization for when the matrix series is regular
based on the algebraic variety of an augmented system.
The structure of the linear system leads to a block triangular system.
In the regular case, solving the linear system is equivalent 
to solving a Hermite interpolation problem.
We show that this solution has cost cubic in the problem size.
In general, at singular points, we rely on methods of tropical
algebraic geometry to compute Puiseux series.
With a few illustrative examples, we demonstrate the
application to polynomial homotopy continuation.
\end{abstract}



\section{Introduction}
\subsection{Preliminaries}

A polynomial homotopy is a family of polynomial systems
which depend on one parameter.  Numerical continuation methods
to track solution paths defined by a homotopy are classical,
see e.g.:~\cite{AG03} and~\cite{Mor87}.
Studies of deformation methods in symbolic computation
appeared in~\cite{BMWW04}, \cite{CPHM01}, and~\cite{HKPSW00}.
In particular, the application of Pad{\'{e}} approximants 
in~\cite{JMSW09} stimulated our development of methods to
compute power series.

\noindent {\bf Problem statement.}  We want to define an efficient,
numerically stable, and robust algorithm to compute a power series 
expansion for a solution curve of a polynomial homotopy.  
The input is a list of polynomials in several variables,
where one of the variables is a parameter denoted by $t$, 
and a value of $t$ near which information is desired.
The output of the algorithm is a tuple of series in $t$ that vanish up to a certain
degree when plugged in to either the original equations or, in special cases, a
transformation of the original equations.

A power series for a solution curve forms the input to the computation
of a Pad{\'{e}} approximant for the solution curve, which will then
provide a more accurate predictor in numerical path trackers.
Polynomial homotopies define deformations of polynomial systems
starting at generic instances and moving to specific instances.
Tracking solution paths that start at singular solutions is not
supported by current numerical polynomial homotopy software systems.
At singular points we encounter series with fractional powers,
Puiseux series.

\noindent {\bf Background and related work.}  As pointed out in~\cite{BM01},
polynomials, power series, and Toeplitz matrices are closely related.  A direct
method to solve block banded Toeplitz systems is presented in~\cite{CV10}.  The
book~\cite{BG96} is a general reference for methods related to approximations
and power series.  We found inspiration for the relationship between
higher-order Newton-Raphson iterations and Hermite interpolation in~\cite{KT74}.
The computation of power series is a classical topic in computer
algebra~\cite{GCL92}.  In~\cite{ACGS04}, new algorithms are proposed 
to manipulate polynomials by values via Lagrange interpolation.

The Puiseux series field is one of the building blocks of tropical algebraic
geometry~\cite{MS15}.  For the leading terms of the Puiseux series,
we rely on tropical methods~\cite{BJSST07}, and in particular on the 
constructive proof of the fundamental theorem of tropical algebraic 
geometry~\cite{JMM08}, see also~\cite{Kat08} and~\cite{Pay09}.
Computer algebra methods for Puiseux series in two dimensions
can be found in~\cite{PR12}.

\noindent {\bf Our contributions.}  Via linearization, rewriting matrices of
series into series with matrix coefficients, we formulate the problem of
computing the updates in Newton's method as a block structured linear algebra
problem.  For matrix series where the leading coefficient is regular, the
solution of the block linear system satisfies the Hermite interpolation problem.
For general matrix series, where several of the leading matrix coefficients may
be rank deficient, Hermite-Laurent interpolation applies. We characterize when
these cases occur using the algebraic variety of an augmented system. To solve the block
diagonal linear system, we propose to reduce the coefficient matrix to a lower
triangular echelon form, and we provide a brief analysis of its cost.

The source code for the algorithm presented in this paper is
archived at github via our accounts {\tt nbliss} and {\tt janverschelde}.

\noindent {\bf Acknowledgments.} We thank the organizers of the ILAS 2016
minisymposium on Multivariate Polynomial Computations and Polynomial Systems,
Bernard Mourrain, Vanni Noferini, and Marc Van Barel, for giving the second
author the opportunity to present this work.  In addition, we are grateful to
the anonymous referee who supplied many helpful remarks.

\subsection{Motivating Example: Pad\'e Approximant}\label{introPade}
One motivation for finding a series solution is that
once it is obtained, 
one can directly compute the associated Pad\'e
approximant, which often has much better convergence properties.
Pad{\'e} approximants~\cite{BG96} are applied in 
symbolic deformation algorithms~\cite{JMSW09}.
In this section we reproduce~\cite[Figure~1.1.1]{BG96}
in the context of polynomial homotopy continuation.
Consider the homotopy
\begin{equation}
   (1-t) (x^2 - 1) + t (3 x^2 - 3/2) = 0.
\end{equation}
The function
$\displaystyle x(t) = \left( \frac{1+t/2}{1+2t} \right)^{1/2}$
is a solution of this homotopy.

Its second order Taylor series at $t = 0$ is
$s(t) = 1 - 3 t/4 + 39 t^2/32 + O(t^2)$.
The Pad\'{e} approximant of degree one in numerator and
denominator is $\displaystyle q(t) = \frac{1 + 7t/8}{1 + 13t/8}$.
In Figure~\ref{figex4pade} we see that the series approximates
the function only in a small interval and then diverges,
whereas the Pad{\'{e}} approximant is more accurate.

\begin{figure}[hbt]
\begin{center}
\includegraphics[scale=0.7]{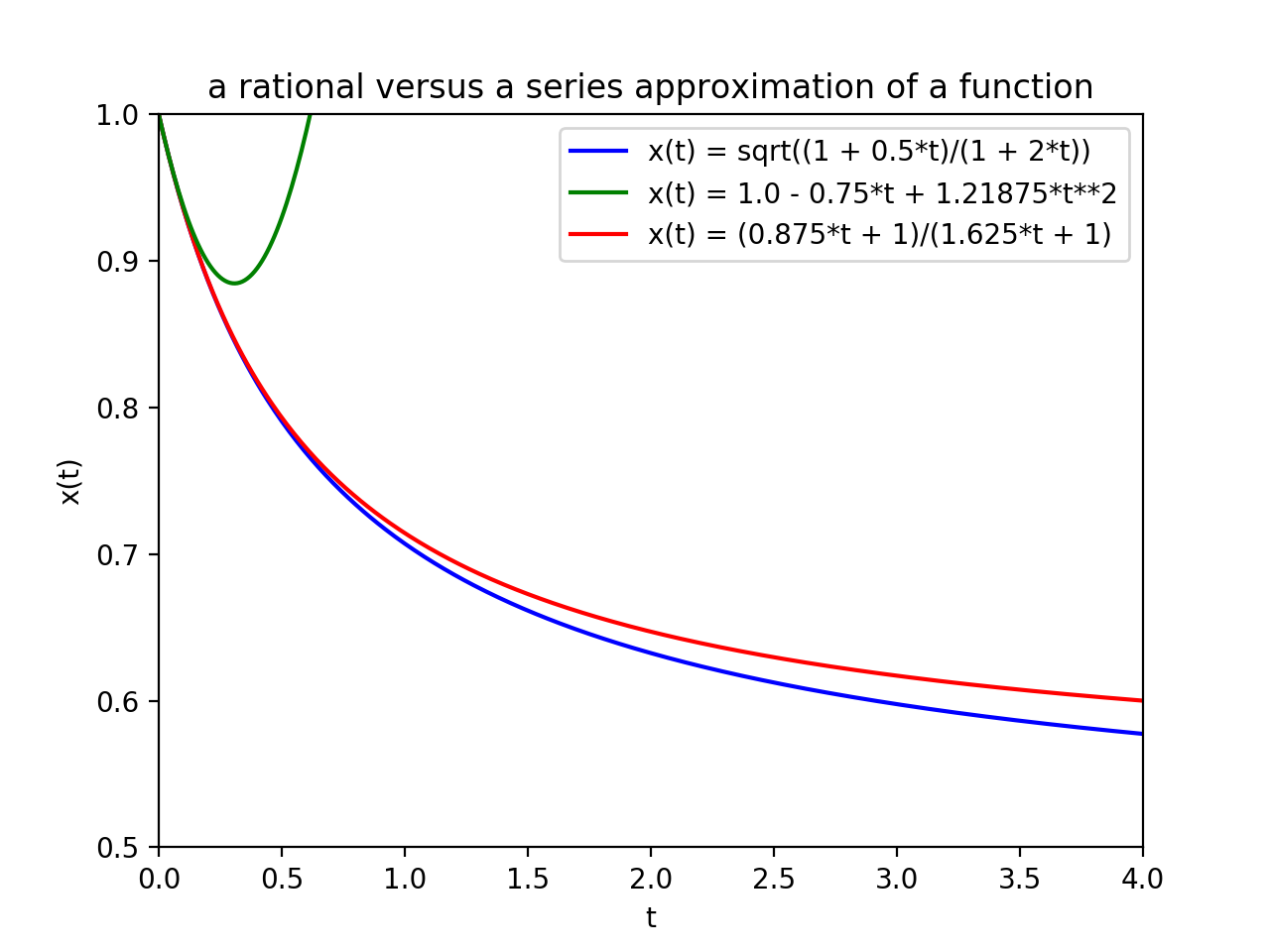}
\caption{Comparing a Pad{\'{e}} approximant to a series approximation
shows the promise of applying Pad{\'{e}} approximants as predictors
in numerical continuation methods.}
\label{figex4pade}
\end{center}
\end{figure}

\subsection{Motivating Example: Viviani's Curve}\label{introViviani}
\begin{figure}[hbtp]
\begin{center}
\includegraphics[scale=.24]{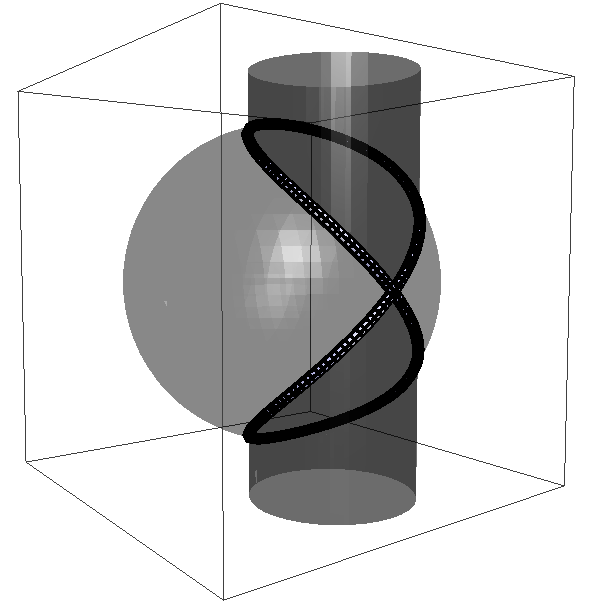}
\caption{Viviani's curve as the intersection of a sphere with a cylinder.}
\label{vivianicurve}
\end{center}
\end{figure}
Viviani's curve is defined as the
intersection of the sphere $(x_1+2)^2 + x_2^2 + x_3^2 = 4$ and the cylinder
$(x_1+1)^2 + x_2^2 = 1$ such that the surfaces are tangent at a single point;
see Figure~\ref{vivianicurve}. 
Our methods will allow us to find the Taylor series expansion around 
any point on a 1-dimensional variety, 
assuming we have suitable starting information. For example, the origin
$(0,0,0)$ satisfies both equations of Viviani's curve. This is the point where
the curve intersects itself, so the curve 
is {\em singular}\footnote{Definition~\ref{def:singular} 
makes this precise for general curves.}
there, meaning algebraically that the Jacobian drops rank, and 
geometrically that the tangent space does not have the expected dimension.
If we apply our methods at this point, we obtain the following
series solution for $x_1,x_2,x_3$:
\begin{equation}\label{vivianiSeries}
 \left\{\begin{array}{l}
   -2t^{2}  \\
   2t - t^{3} - \frac{1}{4}t^{5} - \frac{1}{8}t^{7} - \frac{5}{64}t^{9} -
   \frac{7}{128}t^{11} - \frac{21}{512}t^{13} - \frac{33}{1024}t^{15}  \\
   2t
 \end{array}\right.
\end{equation}
This solution is plotted in Figure~\ref{vivianiApprox} for a varying number of
terms. To check the correctness, we can substitute (\ref{vivianiSeries}) into
the original equations, obtaining series in $O(t^{18})$. The vanishing of the
lower-order terms confirms that we have indeed found an approximate series
solution. Such a solution, possibly transformed into an associated Pad\'e
approximant, would allow for path tracking starting at the origin.

\begin{figure}[hbtp]
\begin{center}
\includegraphics[scale=.24]{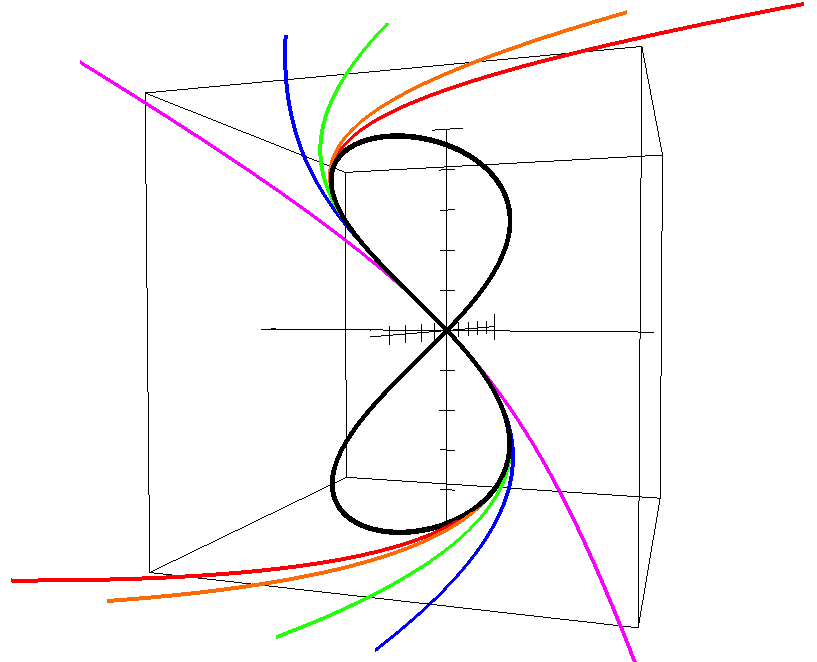}
\caption{Viviani's curve, with improving series approximations
and thus more accurate predictions for points on the curve.}
\label{vivianiApprox}
\end{center}
\end{figure}

\section{The Problem and Our Solution}
\subsection{Problem Setup}\label{sec:problemSetup}
For a polynomial system $\bff = (f_1,f_2,\ldots,f_m)$ where each $f_i \in \CC[t,x_1,\ldots,x_n]$,
the solution variety $\VV(\bff)$ is the set of points $\bfp \in \CC^{n+1}$ such
that $f_1(\bfp)=\cdots=f_m(\bfp)=0$. Let $\bff$ be a system such that
the solution variety is 1-dimensional over $\CC$ and is not contained in the
$t=0$ coordinate hyperplane. We seek to understand
$\VV(\bff)$ by treating the $f_i$'s as elements of $\CC((t))[x_1,\ldots,x_n]$, or in
other words, polynomials in $x_1\ldots x_n$ with coefficients in the ring of
formal Laurent series $\CC((t))$. In this context we will denote the system by
$\tilde{\bff}$.

Our approach is to use Newton iteration on the system
$\tilde{\bff}$. Namely, we find some starting $\z\in \CC((t))^n$ and
repeatedly solve
\begin{equation} \label{eqnewton}
  J_{\tilde{\bff}}(\z) \Delta \z = - \tilde{\bff}(\z)
\end{equation}
for the update $\Delta \z$ to $\z$, where $J_{\tilde{\bff}}$ is the Jacobian
matrix of $\tilde{\bff}$ with respect to $x_1,\ldots,x_n$. This is a system of
equations that is linear over $\CC((t))$, so the problem is well-posed.
Computationally speaking, one approach to solving it would be to overload the
operators on (truncated) power series 
and apply basic linear algebra techniques.
A main point of our paper is that this method can be improved upon.

Of course, applying Newton's method requires a starting guess; here we must
define what it means to be singular:
\begin{definition}\label{def:singular}{\rm 
A point $\bfp$ on a $d$-dimensional component of a variety $\VV(\bff) \subset
\CC^n$ is {\em regular} if the Jacobian of $\bff$ evaluated at $\bfp$ has rank
$n-d$. Points that are not regular are called {\em singular}.}
\end{definition}
In most cases the starting guess for Newton's method can just be a point
$\tilde{\bfp} = (p_1,\ldots,p_n)$ such that $\bfp = (0,p_1,\ldots,p_n)$ is in $\in \VV(\bff)$.
However, if $\bfp$ is a singular point, this is insufficient. In addition, $\bfp$
could be a branch point (which we define later), in which case it is also
not enough to use as the starting guess for Newton's method.

We solve two problems in this paper. First, we find an effective way to perform
the Newton step; the framework is established in
Section~\ref{sec:newtonStep}, and our solution is laid out in
Section~\ref{sec:echelonForm}. And second, we discuss the prelude to
Newton's method in Section~\ref{sec:startingGuess}, characterizing when
techniques from tropical geometry are needed to transform the problem and obtain
the starting guess.

\subsection{The Newton Step}\label{sec:newtonStep}
Solving the Newton step~(\ref{eqnewton}) amounts to solving a linear system 
\begin{equation}\label{basiclinear}
\bf Ax=b
\end{equation}
over the field $\CC((t))$. 
Our first step is linearization, which turns a vector of
series into a series of vectors, and likewise for a matrix series. 
In other words, we refactor the problem and think of $\x$ and $\bfb$ 
as in $\CC^n((t))$
instead of $\CC((t))^n$, and $\bfA$ as in $\CC^{n\times n}((t))$
instead of $\CC((t))^{n\times n}$.

Suppose that $a$ is the lowest order of a term in $\bfA$, and $b$ the lowest
order of a term in $\bfb$. Then we can write the linearized
\begin{align}
\bfA &= A_0t^a + A_1t^{a+1}+\ldots,\\
\bfb &= \bfb_0t^b + \bfb_1t^{b+1} + \ldots, \text{ and}\\
\x   &= \x_0t^{b-a} + \x_1t^{b-a+1}+\ldots
\end{align}
where $A_i\in\CC^{n\times n}$ and
$\bfb_i,\x_i\in\CC^n$. Expanding and equating powers of $t$, the linearized
version of (\ref{basiclinear}) is therefore equivalent to solving
\begin{eqnarray}\label{staggeredSystem}
   A_0 \x_0 & = & \bfb_0 \nonumber \\
   A_0 \x_1 & = & \bfb_1 - A_1 \x_0 \nonumber \\
   A_0 \x_2 & = & \bfb_2 - A_1 \x_1 - A_2 \x_0 \\
            & \vdots & \nonumber  \\
   A_0 \x_d & = & \bfb_d - A_1 \x_{d-1} - A_2 \x_{d-2} - \cdots - A_d \x_0 \nonumber
\end{eqnarray}
for some $d$. This can be written in block matrix form as 
\begin{equation} \label{eqbiglinsys}
  \left[
    \begin{array}{ccccc}
       A_0 &         &         & & \\
       A_1 & A_0     &         & & \\
       A_2 & A_1     & A_0     & & \\
      \vdots & \vdots & \vdots & \ddots & \\
       A_d & A_{d-1} & A_{d-2} & \cdots & A_0 \\
    \end{array}
  \right]
  \left[
    \begin{array}{c}
       \x_0 \\
       \x_1 \\
       \x_2 \\ \vdots \\
       \x_d
    \end{array}
  \right]
  = 
  \left[
    \begin{array}{c}
       \bfb_0 \\
       \bfb_1 \\
       \bfb_2 \\ \vdots \\
       \bfb_d
    \end{array}
  \right].
\end{equation}
For the remainder of this paper, we will use $\z$ and $\Delta \z$ to denote
vectors of series, while $\x$ and $\Delta \x$ will denote their linearized
counterparts, that is, series which have vectors for coefficients.

\begin{example}
{\rm Let 
\begin{equation}
\bff=(2t^2 + tx_1 - x_2 + 1, x_1^3 - 4t^2 + tx_2 + 2t - 1).
\end{equation}
Starting with $\z = (1,1)$, the
first Newton step $J_{\tilde{\bff}}(\z) \Delta \z = -
\tilde{\bff}(\z)$ can be written:
\begin{equation}
   \left[
        \begin{array}{rr}
        t & -1 \\
        3 & t
        \end{array}
   \right]
     \Delta\z
   = 
   -
   \left[
        \begin{array}{r}
         t + 2 t^{2} \\
         3 t - 4 t^{2}
        \end{array}
   \right].
\end{equation}
To put in linearized form, we have $a=0$, $b=1$,
\begin{equation}
A_0 = \left[\begin{array}{rr}
        0 & -1 \\
        3 & 0
        \end{array}\right],
A_1 = \left[\begin{array}{rr}
        1 & 0 \\
        0 & 1
        \end{array}\right],
\end{equation}
\begin{equation}
\bfb_0 = \left[\begin{array}{r}
       -1 \\
       -3 
        \end{array}\right],\text{ and }
\bfb_1 = \left[\begin{array}{r}
       -2 \\
        4 
        \end{array}\right].
\end{equation}
Since $A_0$ is regular, we can solve in staggered form as
in~(\ref{staggeredSystem}), which yields the next term:
\begin{equation}
  \Delta\x = \left[\begin{array}{r}
  - 1 \\
   1
  \end{array}\right]t.
\end{equation}
After another iteration, our series solution is
\begin{equation}\label{regularSol}
  \left[\begin{array}{r}
  1 - t \\
  1 + t + t^{2}
  \end{array}\right].
\end{equation}
In fact this is the entire series solution for
$\bff$ --- substituting~(\ref{regularSol}) 
into $\bff$ causes both polynomials to vanish completely.
}
\end{example}

\begin{remark}{\rm
We constructed the example above so its solution is a series with
finitely many terms, a polynomial.  The solution of~(\ref{basiclinear})
can be interpreted as the solution obtained via Hermite interpolation.
Observe that for a series
\begin{equation}
   \x(t) = \x_0 + \x_1 t + \x_2 t^2 + \x_3 t^3 + \cdots + \x_k t^k + \cdots
\end{equation}
its Maclaurin expansion is
\begin{equation}
   \x(t) = \x(0) + \x'(0) t + \frac{1}{2} \x''(0) t^2
         + \frac{1}{3!} \x'''(0) t^3 + \cdots
         + \frac{1}{k!} \x^{(k)}(0) t^k + \cdots
\end{equation}
where $\x^{(k)}(0)$ denotes 
the $k$-th derivative of $\x(t)$ evaluated at zero.  Then:
\begin{equation}
   \x_k = \frac{1}{k!} \x^{(k)}(0), \quad k=0, 1, \ldots
\end{equation}
Solving~(\ref{basiclinear}) up to degree~$d$ implies that all
derivatives up to degree~$d$ of $\x(t)$ at $t = 0$ match the solution.
If the solution is a polynomial, then this polynomial will be obtained
if~(\ref{basiclinear}) is solved up to the degree of the polynomial.}
\end{remark}

\subsection{The Starting Guess, and Related Considerations}
\label{sec:startingGuess}
Our hope is that a solution $\z(t)$ of $\tilde{\bff}$ parameterizes the curve
in some neighborhood of a point $\bfp \in \VV(\bff)$. In other words, if $\pi$
is the projection map of $\VV(\bff)$ onto the $t$-coordinate axis, then
$\z(t)$ should be a branch of $\pi^{-1}$.  

It is natural to think that there are two scenarios for the starting point $\bfp
\in \VV(\bff)$, namely that it is a regular point or it is singular. And indeed,
when $\bfp$ is singular, tropical methods are required. Intuitively speaking,
when at a singular point, knowing just the point itself is insufficient to
determine the series; higher-derivative information is required. Observe the
second frame of Figure~\ref{fig:nodalCurves}.

The point $\bfp$ being regular, however, is not enough. Consider the third
frame of Figure~\ref{fig:nodalCurves}. Here $x=0$ cannot be lifted because the
origin is a {\em branch point} of the curve. In other words, the derivative at
$\bfp$ in terms of $t$ is undefined, so a Taylor series in $t$ is impossible
without a transformation of the problem.

The proper way to check if Newton's method can be applied directly to
$\bfp$, or whether tropical methods are needed, is by checking if $\bfp$ is a
singular point of $\VV(\bff)\cap \VV(t)$. Setting $\bff_{\aug}=
(t,f_1,\ldots,f_n)$, we have $\VV(\bff_{\aug}) = \VV(\bff)\cap \VV(t)$.
We can thus use $\VV(\bff_{\aug})$ to distinguish the first frame of
Figure~\ref{fig:nodalCurves} from the latter two. This is summarized and proven
in the following.

\begin{figure}
\captionsetup[subfigure]{labelformat=empty}
\centering
\begin{subfigure}{.33\textwidth}
  \centering
\includegraphics[scale=.15]{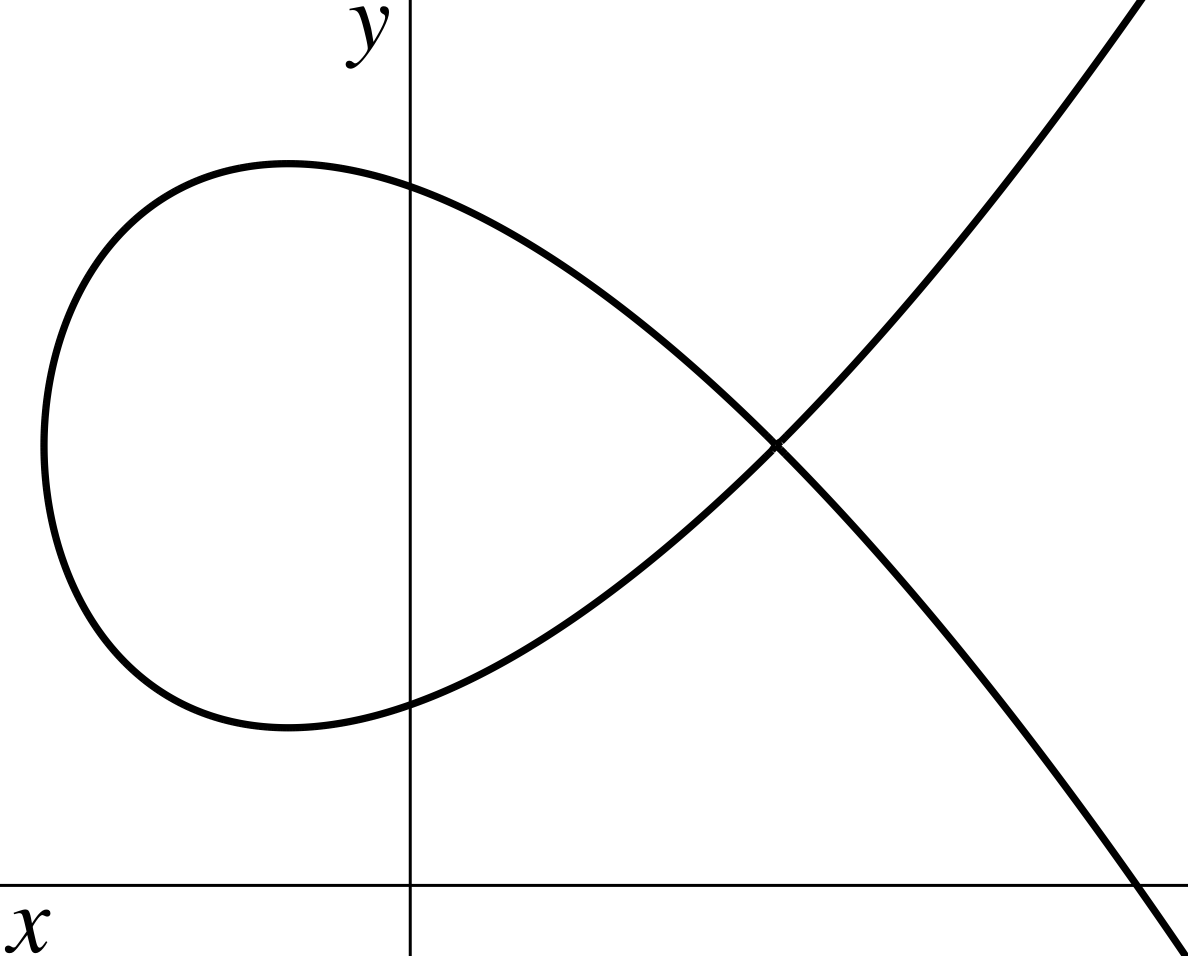}
  \caption{a general point}
\end{subfigure}%
\begin{subfigure}{.33\textwidth}
  \centering
\includegraphics[scale=.15]{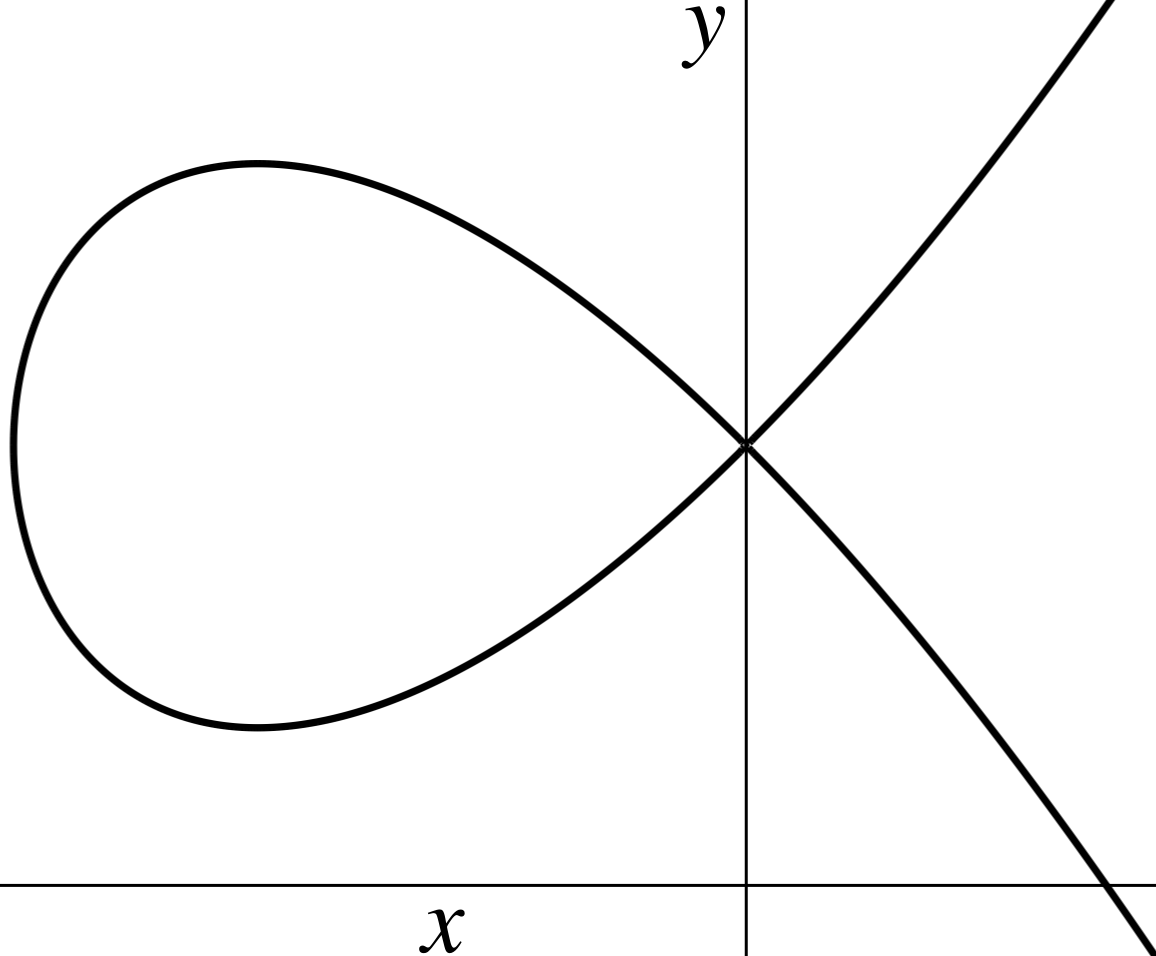}
  \caption{a singularity}
\end{subfigure}
\begin{subfigure}{.33\textwidth}
  \centering
\includegraphics[scale=.15]{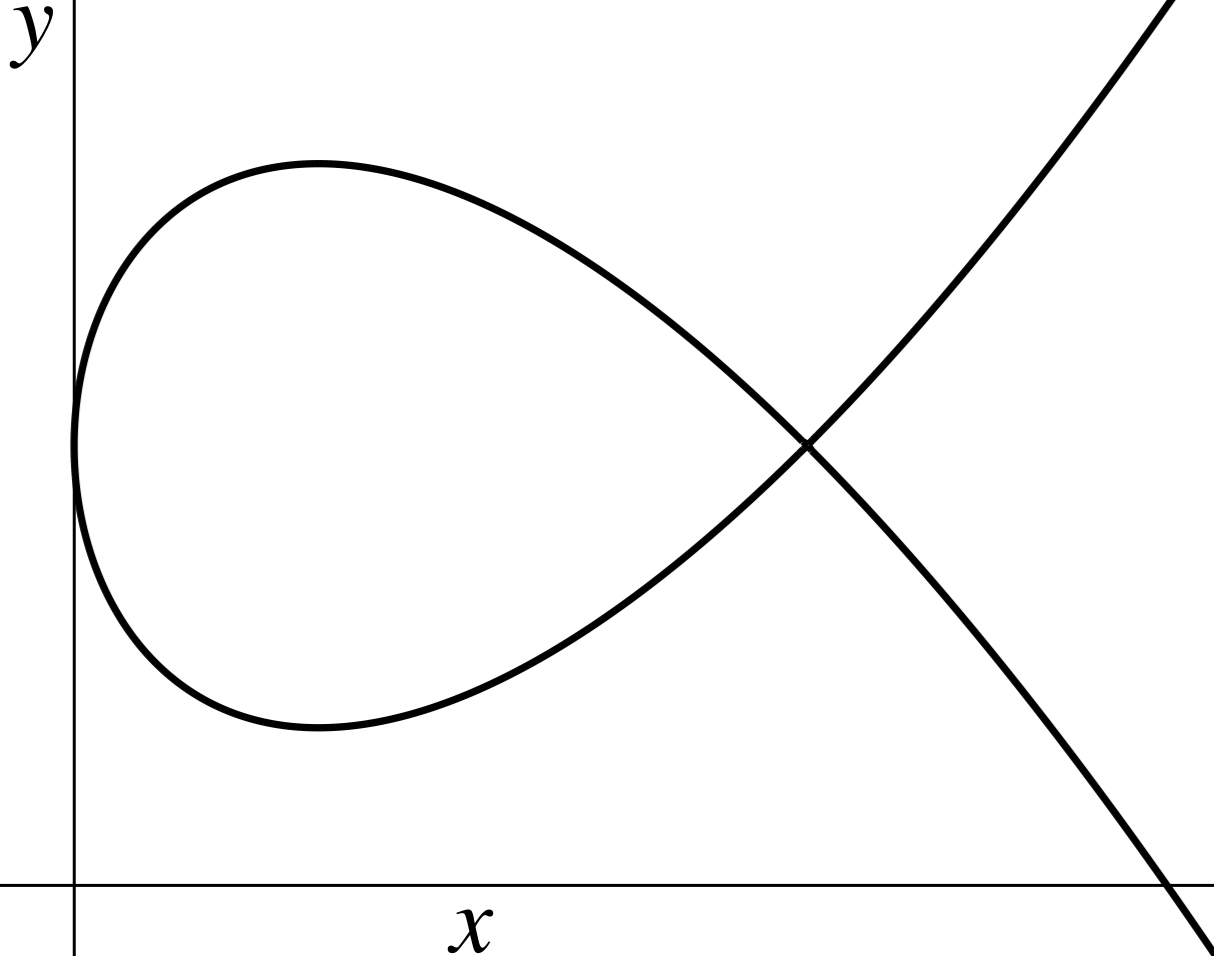}
  \caption{a branch point}
\end{subfigure}
\captionsetup{font=scriptsize}
\caption{\label{fig:nodalCurves}
Lifting $x=0$ to three different types of point.
In general, the line $x=0$ intersects the curve at regular points.
If the curve intersects itself for $x=0$, we are at a singular point.
The curve turns at a branch point.}
\end{figure}

\begin{proposition} \label{regularStartProp}
Let $\bfp = (0,p_1,\ldots,p_n) \in \VV(\bff)$, and set $\tilde{\bfp} =
(p_1,\ldots,p_n)$.  Then $\bfp$ is a regular point of $\VV(\bff_{\aug})$ if and
only if for every step of Newton's method applied to $\x(t):=\tilde{\bfp}$, $a=0$ and
$A_0$ has full rank.
\end{proposition}
\begin{proof}
$(\Rightarrow)$ By definition, $\bfp$ is a regular point of $\bff_{\aug}$ if and
only if $J_{\bff_{\aug}}(\bfp)$ has full rank. But note that $J_{\bff_{\aug}}$ is
\begin{equation}\label{proofJac1}
   \left[
      \begin{array}{cccc}
         1 & 0 & \cdots & 0 \\
         df_1/dt & df_1/dx_1 & \cdots & df_1/dx_n \\
         df_2/dt & df_2/dx_1 & \cdots & df_2/dx_n \\
         \vdots    & \vdots    &        & \vdots \\
         df_m/dt & df_m/dx_1 & \cdots & df_m/dx_n
      \end{array}
   \right].
\end{equation}
and $J_{\tilde{\bff}}$ is
\begin{equation}\label{proofJac2}
   \left[
      \begin{array}{ccc}
         df_1/dx_1 & \cdots & df_1/dx_n \\
         df_2/dx_1 & \cdots & df_2/dx_n \\
         \vdots    &        & \vdots \\
         df_m/dx_1 & \cdots & df_m/dx_n
      \end{array}
   \right].
\end{equation}
So $J_{\bff_{\aug}}$ has full rank at $\bfp$ 
if and only if $J_{\tilde{\bff}}|_{t=0}$ has full rank at $\tilde{\bfp}$. Thus it suffices to show that
after each Newton step, $a=0$ and $\x(0)=\tilde{\bfp}$ remain true, so that $A_0 =
J_{\tilde{\bff}}(\x(0)) = J_{\tilde{\bff}}(\tilde{\bfp})|_{t=0}$ continues to
have full rank.

We clearly have $a\geq 0$ at every step, since the Newton iteration cannot
introduce negative exponents. At the beginning, $a=0$ and $\x(0)=\tilde{\bfp}$ hold
trivially. Inducting on the Newton steps, 
if $a=0$ and $\x(0)=\tilde{\bfp}$ at some point
in the algorithm, then the next $A_0$, 
namely $J_{\tilde{\bff}}(\x(0)) = J_{\tilde{\bff}}(\tilde{\bfp})|_{t=0}$, is the same matrix as
in the last step, hence it is again regular and $a$ is 0. 
Since $\tilde{\bff}(\x(0)) = \tilde{\bff}(\tilde{\bfp})|_{t=0}=0$, $b$ must be strictly 
greater than~0. Thus the next Newton
update $\Delta \x$ must have positive degree in all components,
leaving $\x(0)= \tilde{\bfp}$ unchanged.

$(\Leftarrow)$ If $\bfp$ is a singular point of $\VV(\bff_{\aug})$,
then on the first Newton step $A_0=J_{\tilde{\bff}}(\tilde{\bfp})|_{t=0}$ must drop rank by the
same argument as above comparing~(\ref{proofJac1}) and~(\ref{proofJac2}).
\end{proof}

To summarize the cases:
\begin{lemma}\label{startingScenarios}
There are three possible scenarios for $\VV(\bff_{\aug})$:
\begingroup \addtolength{\jot}{0.5em}
\begin{align}
&\text{1.~~} \exists \bfp \in \VV(\bff_{\aug})\text{ regular,}\nonumber \\
&\text{2.~~} \exists \bfp \in \VV(\bff_{\aug})\text{ singular, or} \nonumber\\
&\text{3.~~} \nexists \bfp \in \VV(\bff_{\aug})\nonumber
\end{align} \endgroup
\end{lemma}
In the first case, we can simply use $\tilde{\bfp} = (p_1,p_2,\ldots,p_n)$ 
to start the Newton iteration. In the
second, we must defer to tropical methods in order to obtain the necessary
starting $\z$, which will lie in $\CC[[t]]^n$. 
In the final case, we also defer to tropical methods, 
which provide a starting $\z$ that will have negative exponents.
A change of coordinates brings the problem back into one of 
the first two cases, and we can apply our method directly.  
It is important to reiterate that $\bfp$ may be a
regular point of $\VV(\bff)$ but a singular point of
$\VV(\bff_{\aug})$, as is the case in the third frame of
Figure~\ref{fig:nodalCurves}. The following example also demonstrates this
behavior.

\begin{example}[Viviani, continued] {\rm
In Section~\ref{introViviani} we introduced the example of Viviani's curve. If
we translate by a substitution so that setting $x_1=0$ gives not the singular point at
the origin, but instead the highest and lowest points on the curve, the system
becomes
\begin{equation}\label{vivianiEqs}
\bff = (x_1^2 + x_2^2 + x_3^2 - 4, (x_1-1)^2 + x_2^2 - 1).
\end{equation}
When $x_1=0$ we obtain the two points $(0,0,2)$ and $(0,0,-2)$, which are both
regular points.
For the augmented system $\bff_{\aug}$, the Jacobian $J_{\bff_{\aug}}$ is
\begin{equation}
    \left[\begin{array}{rrr}
        1 & 0 & 0 \\
        2 x_1 & 2 x_2 & 2 x_3 \\
        2 x_1 - 2 & 2 x_2 & 0
    \end{array}\right]
\end{equation}
which at the point $\bfp = (0,0,2)$ becomes
\begin{equation}
\left[\begin{array}{rrr}
1 & 0 & 0 \\
0 & 0 & 4 \\
-2 & 0 & 0
\end{array}\right].
\end{equation}
This matrix drops rank, hence $\bfp$ is a singular point of $\bff_{\aug}$ and we
are in the second case of Lemma~\ref{startingScenarios}.
Following Lemma~\ref{startingScenarios},
we defer to tropical methods to begin, obtaining
the transformation $x_1\rightarrow 2t^2$ and the starting term
$\z=(2t,2)$. Now the first Newton step can be written:
\begin{equation}\label{vivianiNewton}
   \left[
      \begin{array}{rr}
      4 t & 4 \\
      4 t & 0
      \end{array}
   \right]
     \Delta\z
   = 
   -
   \left[
      \begin{array}{r}
      4 t^{2} + 4 t^{4} \\
      4 t^{4}
      \end{array}
   \right].
\end{equation}
Note that  $J_{\tilde{\bff}}(\z)$ is now invertible over $\CC((t))$.
Its inverse begins with negative exponents of~$t$:
\begin{equation}
   \left[
      \begin{array}{cc}
                  0 & 1/4 \\
         1/4~t^{-1} & -1/4~t^{-1}
      \end{array}
   \right].
\end{equation}

To linearize, we first observe that $a=0$ and $b=2$, 
so $\x$ will have degree at least $b-a=2$.  
The linearized block form of (\ref{vivianiNewton}) is then
\begin{equation}\label{vivianiNewtonLinearized}
   \left[
     \begin{array}{rr|rr|rr}
     0 & 4 & 0 & 0 & 0 & 0 \\
     0 & 0 & 0 & 0 & 0 & 0 \\
     \hline
     4 & 0 & 0 & 4 & 0 & 0 \\
     4 & 0 & 0 & 0 & 0 & 0 \\
     \hline
     0 & 0 & 4 & 0 & 0 & 4 \\
     0 & 0 & 4 & 0 & 0 & 0
     \end{array}
   \right]
     \Delta \x   
   = 
   \left[
     \begin{array}{r}
    -4 \\
     0 \\
     0 \\
     0 \\
    -4 \\
    -4
     \end{array}
   \right].
\end{equation}
Whether we solve (\ref{vivianiNewton}) over $\CC((t))$ or solve
(\ref{vivianiNewtonLinearized}) in the least squares sense, 
we obtain the same Newton update
\begin{equation}
\Delta\x=
\left[\begin{array}{r}
0 \\
-1
\end{array}\right]t^2 + 
\left[\begin{array}{r}
-1 \\
0
\end{array}\right]t^3,
\end{equation}
or in non-linearized form,
\begin{equation}
\Delta\z=
\left[\begin{array}{r}
-t^{3} \\
-t^{2}
\end{array}\right].
\end{equation}
Substituting $\z+\Delta\z = (2t-t^3,2-t^2)$ into (\ref{vivianiEqs}) produces
$(x_1^6 + x_1^4, x_1^6)$, and we have obtained the desired cancellation of
lower-order terms.
}
\end{example}

The matrix in~(\ref{vivianiNewtonLinearized}) 
we call a Hermite-Laurent matrix, because its correspondence
with Hermite-Laurent interpolation.

\subsection{A Lower Triangular Echelon Form}\label{sec:echelonForm}

When we are in the regular case of Lemma~\ref{startingScenarios} and the
condition number of $A_0$ is low, we can
simply solve the staggered system~(\ref{staggeredSystem}). When this is not
possible, we are forced to solve~(\ref{eqbiglinsys}).
Figure~\ref{figechelon} shows the structure of 
the coefficient matrix~(\ref{eqbiglinsys}) for the regular case,
when $A_0$ is regular and all block matrices are dense.
The essence of this section is that we can use column operations to reduce the
block matrix to a lower triangular
echelon form as shown at the right of Figure~\ref{figechelon},
solving~(\ref{eqbiglinsys}) in the same time as~(\ref{staggeredSystem}).

\begin{figure}
\begin{center}
\begin{picture}(400,100)(0,0)
\put(0,0){\includegraphics[scale=.30]{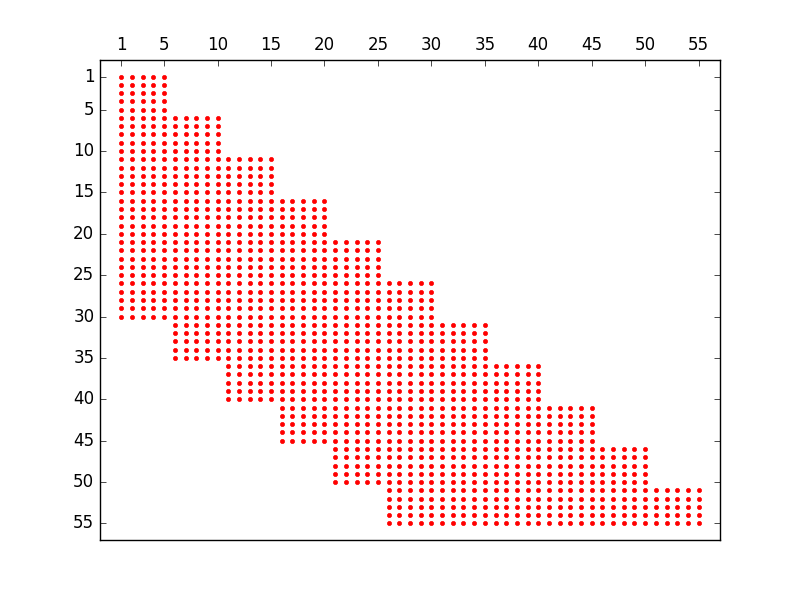}}
\put(170,0){\includegraphics[scale=.30]{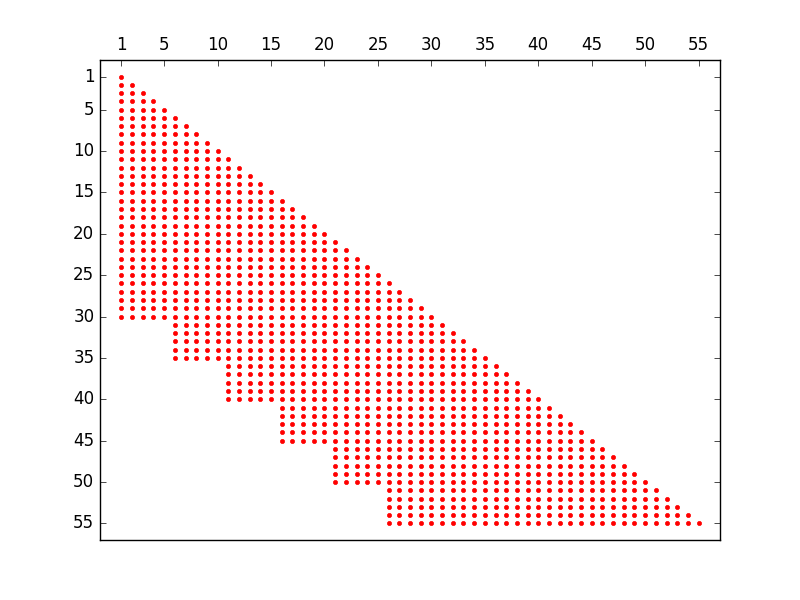}}
\end{picture}
\caption{The banded block structure of a generic Hermite-Laurent matrix
for $n=5$ at the left, with at the right its lower triangular echelon form.}
\label{figechelon}
\end{center}
\end{figure}

The lower triangular echelon form of a matrix is a lower triangular
matrix with zero elements above the diagonal.
If the matrix is regular, then all diagonal elements are nonzero.
For a singular matrix, the zero rows of its echelon form are on top
(have the lowest row index) and the zero columns are at the right
(have the highest column index).
Every nonzero column has one pivot element, which is the 
nonzero element with the smallest row index in the column.
All elements at the right of a pivot are zero.
Columns may need to be swapped so that the row indices of the
pivots of columns with increasing column indices are sorted
in decreasing order.

\begin{example} {\rm (Viviani, continued).
For the matrix series in~(\ref{vivianiNewtonLinearized}),
we have the following reduction:
\begin{equation}
  \left[
     \begin{array}{cccccc}
        0 & 4 & 0 & 0 & 0 & 0 \\
        0 & 0 & 0 & 0 & 0 & 0 \\
        4 & 0 & 0 & 4 & 0 & 0 \\
        4 & 0 & 0 & 0 & 0 & 0 \\
        0 & 0 & 4 & 0 & 0 & 4 \\
        0 & 0 & 4 & 0 & 0 & 0
     \end{array}
  \right]
  \rightarrow
  \left[
     \begin{array}{cccccc}
        0 & 0 &  0 & 0 &  0 & 0 \\
        4 & 0 &  0 & 0 &  0 & 0 \\
        0 & 4 &  0 & 0 &  0 & 0 \\
        0 & 4 &  4 & 0 &  0 & 0 \\
        0 & 0 &  0 & 4 &  0 & 0 \\
        0 & 0 &  0 & 4 &  4 & 0
     \end{array}
  \right].
\end{equation}
Because of the singular matrix coefficients in the series,
we find zeros on the diagonal in the echelon form.
}
\end{example}

Given a general $n$-by-$m$ dimensional matrix~$A$, 
the lower triangular echelon form $L$
can be described by one $n$-by-$n$ row permutation matrix $P$ which
swaps the zero rows of~$A$ and a sequence of $m$
column permutation matrices $Q_k$ (of dimension $m$) and
multiplier matrices $U_k$ (also of dimension $m$).
The matrices $Q_k$ define the column swaps to bring the pivots
with lowest row indices to the lowest column indices.
The matrices $U_k$ contain the multipliers to reduce what
is at the right of the pivots to zero.
Then the construction of the lower triangular echelon form
can be summarized in the following matrix equation:
\begin{equation}
   L = P A Q_1 U_1 Q_2 U_2 \cdots Q_m U_m.
\end{equation}
Similar to solving a linear system with a LU factorization,
the multipliers are applied to the solution of the lower
triangular system which has $L$ as its coefficient matrix.

\section{Some Preliminary Cost Estimates}

Working with truncated power series is somewhat similar to working
with extended precision arithmetic.
In this section we make some observations regarding the cost overhead.

\subsection{Cost of one step}
First we compare the cost of computing a single Newton step using the various
methods introduced. We let $d$ denote the degree of the truncated series in $\bfA(t)$,
and $n$ the dimension of the matrix coefficients in $\bfA(t)$ as before.

\noindent {\bf The staggered system.} In the case that
$a\geq 0$ and the leading coefficient $A_0$ of the matrix series $\bfA(t)$ is
regular, the equations in~{\rm (\ref{staggeredSystem})} can be solved with
$O(n^3) + O(d n^2)$ operations.  The cost is $O(n^3)$ for the decomposition of
the matrix $A_0$, and $O(d n^2)$ for the back substitutions using the
decomposition of $A_0$ and the convolutions to compute the right hand sides.

\noindent {\bf The big block matrix.} Ignoring the
triangular matrix structure, the cost of solving the larger linear
system~(\ref{eqbiglinsys}) is $O((dn)^3)$.

\noindent {\bf The lower triangular echelon version.}
If the leading coefficient $A_0$ in the matrix series is regular
(as illustrated by Figure~\ref{figechelon}), we may copy the
lower triangular echelon form $L_0 = A_0 Q_0 U_0$ of $A_0$ to all blocks 
on the diagonal and apply the permutation $Q_0$ and column operations
as defined by $U_0$ to all other column blocks in~$\bfA$.
The regularity of~$A_0$
implies that we may use the lower triangular echelon form of $L_0$
to solve~(\ref{eqbiglinsys}) with substitution.
Thus with this quick optimization we obtain
the same cost as solving the staggered system~(\ref{staggeredSystem}).

In general, $A_0$ and several other matrix coefficients
may be rank deficient, and the diagonal of nonzero pivot elements will
shift towards the bottom of~$L$.  We then find as solutions vectors
in the null space of the upper portion of the matrix~$\bfA$.

\subsection{Cost of computing $D$ terms}
Assume that $D=2^k$. In the regular case, assuming quadratic convergence, it
will take $k$ steps to compute $2^k$ terms. We can reuse the factorization of
$A_0$ at each step, so we have $O(n^3)$ for the decomposition plus 
\begin{equation}
O(2n^2 + 4n^2 + 8n^2 + \cdots + 2^{k-1}n^2) = O(2^kn^2)
\end{equation}
for the back substitutions. Putting these together,
we find the cost of computing $D$ terms to be $O(n^3) + O(D n^2)$.

\section{Computational Experiments}

Our power series methods have been implemented in PHCpack~\cite{Ver99}
and are available to the Python programmer via phcpy~\cite{Ver14}.
To set up the problems we used the computer algebra system Sage~\cite{Sage}, and
for tropical computations we used Gfan~\cite{BHS08} and Singular~\cite{DGPS} via
the Sage interface.

\subsection{The Problem of Apollonius}

\begin{figure}[hbt]
\begin{center}
\includegraphics[scale=.30]{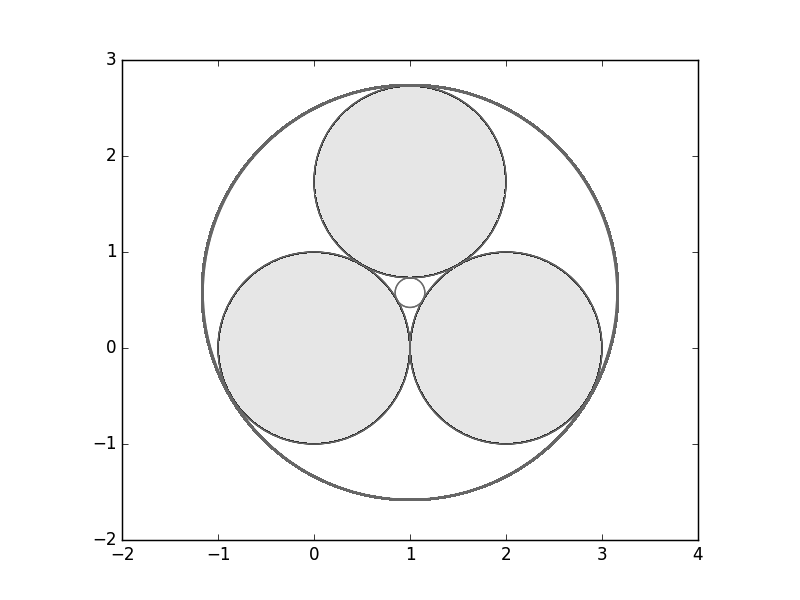}
\caption{Singular configuration of Apollonius circles. 
The input circles are filled in, the solution circles are dark gray.
Because the input circles mutually touch each other,
three of the solution circles coincide with the input circles. }
\label{appolfig}
\end{center}
\end{figure}

The classical problem of Apollonius consists in finding all circles that are
simultaneously tangent to three given circles. 
A special case is when the three
circles are mutually tangent and have the same radius; see
Figure~\ref{appolfig}. Here the solution variety is singular -- the circles
themselves are double solutions. In this figure, all have radius 3, and
centers $(0,0)$, $(2,0)$, and $(1,\sqrt{3})$. We can study this
configuration with power series techniques by introducing a parameter $t$ to
represent a vertical shift of the upper circle. We then examine the solutions
as we vary $t$. This is represented algebraically as a solution to 
\begin{equation} \label{appollo}
   \left\{
      \begin{array}{rcl}
        x_1^2 + x_2^2 - r^2 - 2r - 1 & = & 0 \\
        x_1^2 + x_2^2 - r^2 - 4x_1 - 2r + 3 & = & 0 \\
        t^2 + x_1^2 - 2tx_2 + x_2^2 - r^2 
        + 2\sqrt{3}t - 2x_1 - 2\sqrt{3}x_2 + 2r + 3 & = & 0.
      \end{array}
   \right.
\end{equation}
Because we are interested in power series solutions of~(\ref{appollo})
near $t=0$, we use $t$ as our free variable. 
To simplify away the $\sqrt{3}$, we substitute
$t\rightarrow \sqrt{3}t$, $x_2 \rightarrow \sqrt{3} x_2$, 
and the system becomes
\begin{equation} \label{appollo2}
   \left\{
      \begin{array}{rcl}
        x_1^2 + 3x_2^2 - r^2 - 2r - 1 & = & 0 \\
        x_1^2 + 3x_2^2 - r^2 - 4x_1 - 2r + 3 & = & 0 \\
        3t^2 + x_1^2 - 6tx_2 + 3x_2^2
        - r^2 + 6t - 2x_1 - 6x_2 + 2r + 3 & = & 0.
      \end{array}
   \right.
\end{equation}
Call this system $\bff$.
Now we examine the system at $(t,x_1,x_2,r) = (0,1,1,1) = \bfp$.
The Jacobian $J_{\bff}$ at $\bfp$ is
\begin{equation}
  \left[\begin{array}{rrrr}
    0 & 2 & 6 & -4 \\
    0 & -2 & 6 & -4 \\
    0 & 0 & 0 & 0
  \end{array}\right],
\end{equation}
so $\bff$ --- and by extension $\bff_{\aug}$ --- is singular at $\bfp$,
and we are in the second case of Lemma~\ref{startingScenarios}.
Tropical methods give two possible starting solutions,
which rounded for readability are
$(t,1,1+0.536t,1+0.804t)$ and $(t,1,1 + 7.464t,1 + 11.196t)$.
We will continue with the second; call it $\z$.
For the first step of Newton's method, $\bfA$ is
\begin{equation}
\left[\begin{array}{rrr}
     2 &  6  &  -4 \\
    -2 &  6  &  -4 \\
     0 &  0  &   0
\end{array}\right] + 
\left[\begin{array}{rrr}
      0  &  44.785  &  -22.392 \\
      0  &  44.785  &  -22.392 \\
      0  &  38.785  &  -22.392
\end{array}\right]t
\end{equation}
and $\bfb$ is
\begin{equation}
\left[\begin{array}{r}
  41.785 \\
  41.785 \\
  0
\end{array}\right]t^2.
\end{equation}
From these we can construct the linearized system
\begin{equation}
\left[\begin{array}{rrr}
    A_0 &     & \\
    A_1 & A_0 & \\
        & A_1 & A_0 \\
\end{array}\right]\Delta \x = 
\left[\begin{array}{c}
    \bfb_0 \\
    0 \\
    0
\end{array}\right].
\end{equation}
Solving in the least squares sense, we obtain two more terms of the series,
so in total we have
\begin{equation} \label{appolloSol2}
   \left\{
      \begin{array}{rcl}
        x_1 & = &1 \\
        x_2 & = & 1 + 7.464t + 45.017t^2 + 290.992t^3 \\
          r & = & 1 + 11.196t + 77.971t^2 + 504.013t^3.
      \end{array}
   \right.
\end{equation}
By comparison, the series we obtain from the other possible starting solution is
\begin{equation} \label{appolloSol3}
   \left\{
      \begin{array}{rcl}
        x_1 & = &1 \\
        x_2 & = &  1 + 0.536t - 0.017t^2 + 0.0077t^3\\
          r & = &  1 + 0.804t + 0.029t^2 - 0.013t^3.
      \end{array}
   \right.
\end{equation}
From these, we get a good idea of what happens near $t=0$:
the first solution circle grows rapidly
(corresponding to the larger coefficients in~(\ref{appolloSol2})),
while the other stays small
(corresponding to the smaller coefficient in~(\ref{appolloSol3})).
This is illustrated in Figure~\ref{shiftedAppol},
which shows the solutions of the system at $t=0.13$.

\begin{figure}[hbt]
\begin{center}
\includegraphics[scale=.55]{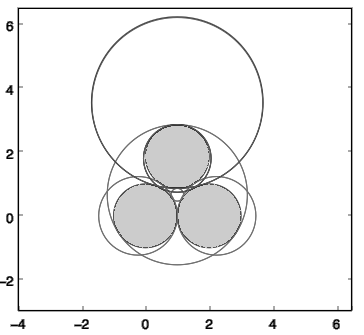}
\caption{Solution to (\ref{appollo}) for $t=0.13$.
The largest circles correspond to power series solutions with larger
coefficients than the coefficients of the power series solutions
for the smaller circles.}
\label{shiftedAppol}
\end{center}
\end{figure}

This example demonstrates the application of power series solutions
in polynomial homotopies.  
Current numerical continuation methods cannot be applied to track 
the solution paths defined by the homotopy in~(\ref{appollo}),
because at $t=0$, the start solutions are double solutions.
The power series solutions provide reliable predictions
to start tracking the solution paths defined by~(\ref{appollo}).

\subsection{Tangents to Four Spheres}

Our next example is that of finding all lines mutually tangent to four
spheres in $\mathbb{R}^3$; see~\cite{Dur98}, \cite{MPT01}, \cite{Sot11},
and~\cite{ST08}.
If a sphere $S$ has center $\bfc$ and
radius $r$, the condition that a line in $\mathbb{R}^3$ is tangent to
$S$ is given by
\begin{equation}
\|\bfm - \bfc \times \bft\|^2 - r^2 = 0,
\end{equation}
where $\bfm=(x_0,x_1,x_2)$ and $\bft=(x_3,x_4,x_5)$ are the moment and
tangent vectors of the line, respectively. For four spheres, this
gives rise to four polynomial equations; if we add the equation
$x_0x_3 + x_1x_4 + x_2x_5 = 0$ to require that $\bft$ and $\bfm$ are
perpendicular and $x_3^2 + x_4^2 + x_5^2 = 1$ to require that
$\|\bft\| = 1$, we have a system of 6 equations in 6 unknowns which we
expect to be 0-dimensional.

\begin{figure}[hbtp]
\begin{center}
\includegraphics[scale=.30]{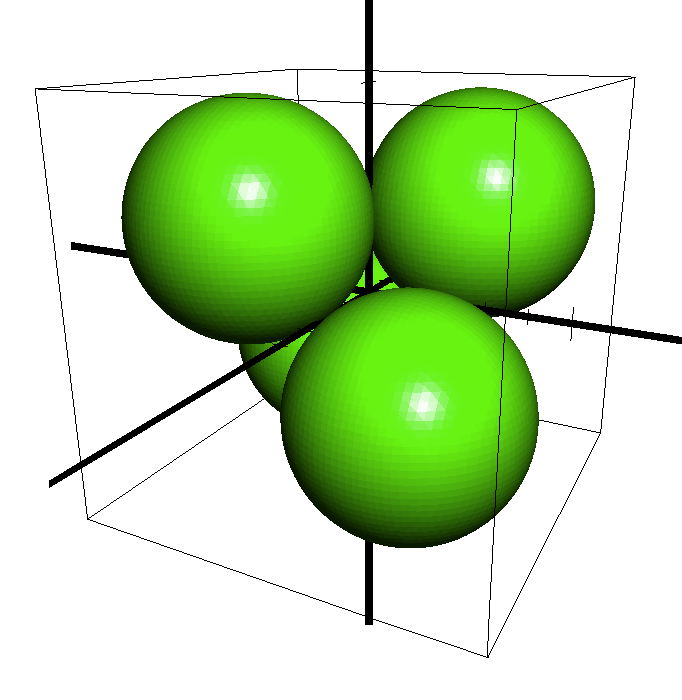}
\caption{A singular configuration of four spheres.
The input spheres mutually touch each other
and the tangent lines common to all four
input spheres occur with multiplicity.}
\label{singularSpheres}
\end{center}
\end{figure}

If we choose the centers to be $(+1,+1,+1)$, $(+1,-1,-1)$,
$(-1,+1,-1)$, and $(-1,-1,+1)$ and the radii to all be $\sqrt{2}$, the
spheres all mutually touch and the configuration is singular; see
Figure~\ref{singularSpheres}. In this case, the number of solutions
drops to three, each of multiplicity 4.

Next we introduce an extra parameter $t$ to the equations so that the radii of
the spheres are $\sqrt{2}+t$. This results in a 1-dimensional system $F$, which
we omit for succinctness. $F$ is singular at $t=0$, so we are once again in the
second case of Lemma~\ref{startingScenarios}.
Tropical and algebraic techniques --- in particular, the tropical
basis~\cite{BHS08} in Gfan~\cite{Jen08} and the primary decomposition
in Singular~\cite{DGPS} ---
decompose $F$ into three systems, one of which is
\begin{equation} \label{fourSpheresEq}
   \bff=
   \left\{
      \begin{array}{rcl}
        x_0 &= &0 \\
        x_3 &= &0 \\
        x_4^2 + x_2x_5 + x_5^2 &= &0 \\
        x_1x_4 + x_2x_5 &= &0 \\
        x_1x_2 - x_2x_4 + x_1x_5 &= &0 \\
        x_1^2 + x_2^2 - 1 &= &0 \\
        2t^4 + 4t^2 + x_2x_5 &= &0 \\
        x_2^2x_4 - x_2x_4x_5 + x_1x_5^2 - x_4 &= &0 \\
        x_2^3 - x_2 - x_5 &= &0.
      \end{array}
   \right.
\end{equation}
Using our methods we can find several solutions to this, one of which is
\begin{equation}\nonumber
 \left\{
  \begin{array}{l}
    x_0 = 0 \\
    x_1 = 2t + 4.5t^3 + 30.9375t^5 + 299.3906t^7 + 3335.0889t^9 + 40316.851t^{11} \\
    x_2 = 1 - 2t^2 - 11t^4 - 94t^6 - 986.5t^8 - 11503t^{10} \\
    x_3 = 0 \\
    x_4 = 2t - 3.5t^3 - 23.0625t^5 - 193.3594t^7 - 2019.3486t^9 - 23493.535t^{11} \\
    x_5 = -4t^2 - 10t^4 - 64t^6 - 614t^8 - 6818t^{10} - 82283t^{12}
  \end{array}.
 \right.
\end{equation}
Substituting back into $\bff$ yields series in $O(t^{12})$, confirming
the calculations. This solution could be used as the initial predictor
in a homotopy beginning at the singular configuration.

In contrast to the small Apollonius circle problem,
this example is computationally more challenging,
as covered in~\cite{Dur98}, \cite{MPT01}, \cite{Sot11}, and~\cite{ST08}.
It illustrates the combination of tropical methods in
computer algebra with symbolic-numeric power series computations
to define a polynomial homotopy to track solution paths starting
at multiple solutions.

\subsection{Series Developments for Cyclic 8-Roots}

A vector $\bfu \in \CC^n$ of a unitary matrix $A$ is biunimodular
if for $k=1, 2, \ldots, n$: $|u_k| = 1$ and $|v_k| = 1$ 
for $\bfv = A \bfu$.
The following system arises in the study~\cite{FR15}
of biunimodular vectors:

\begin{equation}
  \bff(\x) = 
  \left\{
   \begin{array}{c}
     x_{0}+x_{1}+ \cdots +x_{n-1}=0 \\
     i = 2, 3, 4, \ldots, n-1: 
      \displaystyle\sum_{j=0}^{n-1} ~ \prod_{k=j}^{j+i-1}
      x_{k~{\rm mod}~n}=0 \\
     x_{0}x_{1}x_{2} \cdots x_{n-1} - 1 = 0. \\
    \end{array}
 \right.
\end{equation}

Cyclic 8-roots has solution curves not reported by Backelin~\cite{Bac89}.
Note that because of the last equation, the system has no
solution for $x_0=0$, or in other words $\VV(\bff_{\aug})=\emptyset$. Thus
we are in the third case of Lemma~\ref{startingScenarios}.

In~\cite{AV12,AV13}, the vector $\bfv = (1, -1, 0, 1, 0, 0, -1, 0)$
gives the leading exponents of the series.
The corresponding unimodular coordinate transformation $\x = \z^M$ is
\begin{equation}
  M = 
  \left[
    \begin{array}{rrrrrrrr}
       1 & -1 & 0 & 1 & 0 & 0 & -1 & 0 \\
       0 & 1 & 0 & 0 & 0 & 0 & 0 & 0 \\
       0 & 0 & 1 & 0 & 0 & 0 & 0 & 0 \\
       0 & 0 & 0 & 1 & 0 & 0 & 0 & 0 \\
       0 & 0 & 0 & 0 & 1 & 0 & 0 & 0 \\
       0 & 0 & 0 & 0 & 0 & 1 & 0 & 0 \\
       0 & 0 & 0 & 0 & 0 & 0 & 1 & 0 \\
       0 & 0 & 0 & 0 & 0 & 0 & 0 & 1 \\
    \end{array}
 \right]
 \quad
 \begin{array}{l}
    x_{0} \rightarrow z_{0} \\
    x_{1} \rightarrow z_{1} z_{0}^{-1} \\
    x_{2} \rightarrow z_{2} \\
    x_{3} \rightarrow z_{3} z_{0} \\
    x_{4} \rightarrow z_{4} \\
    x_{5} \rightarrow z_{5} \\
    x_{6} \rightarrow z_{6} z_{0}^{-1} \\
    x_{7} \rightarrow z_{7}.
\end{array}
\end{equation}
Solving the transformed system with $z_0$ set to $0$ gives
the leading coefficient of the series.

After 2 Newton steps, invoked in PHCpack with {\tt phc -u},
the series for $z_1$ is

{\small
\begin{verbatim}
(-1.25000000000000E+00 + 1.25000000000000E+00*i)*z0^2
+( 5.00000000000000E-01 - 2.37676980513323E-17*i)*z0
+(-5.00000000000000E-01 - 5.00000000000000E-01*i);
\end{verbatim}
}

After a third step, the series for $z_1$ is

{\small
\begin{verbatim}
( 7.12500000000000E+00 + 7.12500000000000E+00*i)*z0^4
+(-1.52745512076048E-16 - 4.25000000000000E+00*i)*z0^3
+(-1.25000000000000E+00 + 1.25000000000000E+00*i)*z0^2
+( 5.00000000000000E-01 - 1.45255178343636E-17*i)*z0
+(-5.00000000000000E-01 - 5.00000000000000E-01*i);
\end{verbatim}
}

Bounds on the degree of the Puiseux series expansion
to decide whether a point is isolated are derived in~\cite{HSJ16}.
While the explicit bounds (which can be computed without prior
knowledge of the degrees of the solution curves) are large,
the test of whether a point is isolated can still be
performed efficiently with our quadratically convergent Newton's method.

In a future work, we plan to apply the power series methods to
the cyclic 16-roots problem, the 16-dimensional version of this
polynomial system, for which the tropical prevariety was computed
recently~\cite{JSV17}.



\begin{thebibliography}{10}

\bibitem{AV12}
D.~Adrovic and J.~Verschelde.
\newblock Computing {P}uiseux series for algebraic surfaces.
\newblock In J.~van~der Hoeven and M.~van Hoeij, editors, {\em Proceedings of
  the 37th International Symposium on Symbolic and Algebraic Computation (ISSAC
  2012)}, pages 20--27. ACM, 2012.

\bibitem{AV13}
D.~Adrovic and J.~Verschelde.
\newblock Polyhedral methods for space curves exploiting symmetry applied to
  the cyclic $n$-roots problem.
\newblock In V.P. Gerdt, W.~Koepf, E.W. Mayr, and E.V. Vorozhtsov, editors,
  {\em Computer Algebra in Scientific Computing, 15th International Workshop,
  CASC 2013, Berlin, Germany}, volume 8136 of {\em Lecture Notes in Computer
  Science}, pages 10--29, 2013.

\bibitem{AG03}
E.~L. Allgower and K.~Georg.
\newblock {\em Introduction to Numerical Continuation Methods}, volume~45 of
  {\em Classics in Applied Mathematics}.
\newblock SIAM, 2003.

\bibitem{ACGS04}
A.~Amiraslani, R.~M. Corless, L.~Gonzalez-Vega, and A.~Shakoori.
\newblock Polynomial algebra by values.
\newblock Technical report, Ontario Research Centre for Computer Algebra, 2004.

\bibitem{Bac89}
J.~Backelin.
\newblock Square multiples n give infinitely many cyclic n-roots.
\newblock Reports, Matematiska Institutionen~8, Stockholms universitet, 1989.

\bibitem{BG96}
G.~A. Baker and P.~Graves-Morris.
\newblock {\em {Pad{\'{e}} Approximants}}, volume~59 of {\em Encyclopedia of
  Mathematics and its Applications}.
\newblock Cambridge University Press, 2nd edition, 1996.

\bibitem{BM01}
D.~A. Bini and B.~Meini.
\newblock Solving block banded block {T}oeplitz systems with structured blocks:
  Algorithms and applications.
\newblock In D.~A. Bini, E.~Tyrtyshnikov, and P.~Yalamov, editors, {\em
  Structured Matrices}, pages 21--41. Nova Science Publishers, Inc., Commack,
  NY, USA, 2001.

\bibitem{BHS08}
T.~Bogart, M.~Hampton, and W.~Stein.
\newblock {\tt groebner\_fan} module of {S}age.
\newblock {The Sage Development Team, 2008.}

\bibitem{BJSST07}
T.~Bogart, A.~N. Jensen, D.~Speyer, B.~Sturmfels, and R.~R. Thomas.
\newblock Computing tropical varieties.
\newblock {\em Journal of Symbolic Computation}, 42(1):54--73, 2007.

\bibitem{BMWW04}
A.~Bompadre, G.~Matera, R.~Wachenchauzer, and A.~Waissbein.
\newblock Polynomial equation solving by lifting procedures for ramified
  fibers.
\newblock {\em Theoretical Computer Science}, 315(2-3):335--369, 2004.

\bibitem{CPHM01}
D.~Castro, L.~M. Pardo, K.~H{\"{a}}gele, and J.~E. Morais.
\newblock Kronecker's and {N}ewton's approaches to solving: A first comparison.
\newblock {\em Journal of Complexity}, 17(1):212--303, 2001.

\bibitem{CV10}
A.~Chesnokov and M.~Van~Barel.
\newblock A direct method to solve block banded block {T}oeplitz systems with
  non-banded {T}oeplitz blocks.
\newblock {\em Journal of Computational and Applied Mathematics},
  234(5):1485--1491, 2010.

\bibitem{DGPS}
W.~Decker, G.-M. Greuel, G.~Pfister, and H.~Sch\"onemann.
\newblock {\sc Singular} {4-1-0} --- {A} computer algebra system for polynomial
  computations.
\newblock \url{http://www.singular.uni-kl.de}, 2016.

\bibitem{Dur98}
C.~Durand.
\newblock {\em Symbolic and Numerical Techniques for Constraint Solving}.
\newblock PhD thesis, Purdue University, 1998.

\bibitem{FR15}
H.~F{\"{u}}hr and Z.~Rzeszotnik.
\newblock On biunimodular vectors for unitary matrices.
\newblock {\em Linear Algebra and its Applications}, 484:86--129, 2015.

\bibitem{GCL92}
K.~O. Geddes, S.~R. Czapor, and G.~Labahn.
\newblock {\em Algorithms for Computer Algebra}.
\newblock Kluwer Academic Publishers, 1992.

\bibitem{HKPSW00}
J.~Heintz, T.~Krick, S.~Puddu, J.~Sabia, and A.~Waissbein.
\newblock Deformation techniques for efficient polynomial equation solving.
\newblock {\em Journal of Complexity}, 16(1):70--109, 2000.

\bibitem{HSJ16}
M.~I. Herrero, G.~Jeronimo, and J.~Sabia.
\newblock Puiseux expansions and nonisolated points in algebraic varieties.
\newblock {\em Communications in Algebra}, 44(5):2100--2109, 2016.

\bibitem{JSV17}
A.~Jensen, J.~Sommars, and J.~Verschelde.
\newblock Computing tropical prevarieties in parallel.
\newblock In H.-W. Loidl, M.~Monagan, and J.-C. Faug{\`{e}}re, editors, {\em
  Proceedings of the International Workshop on Parallel Symbolic Computation
  (PASCO 2017)}. ACM, 2017.

\bibitem{Jen08}
A.~N. Jensen.
\newblock Computing {G}r{\"{o}}bner fans and tropical varieties in {G}fan.
\newblock In M.E. Stillman, N.~Takayama, and J.~Verschelde, editors, {\em
  Software for Algebraic Geometry}, volume 148 of {\em The IMA Volumes in
  Mathematics and its Applications}, pages 33--46. Springer-Verlag, 2008.

\bibitem{JMM08}
A.~N. Jensen, H.~Markwig, and T.~Markwig.
\newblock An algorithm for lifting points in a tropical variety.
\newblock {\em Collectanea Mathematica}, 59(2):129--165, 2008.

\bibitem{JMSW09}
G.~Jeronimo, G.~Matera, P.~Solern{\'{o}}, and A.~Waissbein.
\newblock Deformation techniques for sparse systems.
\newblock {\em Found.\ Comput.\ Math.}, 9:1--50, 2009.

\bibitem{Kat08}
E.~Katz.
\newblock A tropical toolkit.
\newblock {\em Expositiones Mathematicae}, 27:1--36, 2009.

\bibitem{KT74}
H.~T. Kung and J.~F. Traub.
\newblock Optimal order of one-point and multipoint iteration.
\newblock {\em Journal of the Association of Computing Machinery},
  21(4):643--651, 1974.

\bibitem{MPT01}
I.~G. Macdonald, J.~Pach, and T.~Theobald.
\newblock Common tangents to four unit balls in $\rr^3$.
\newblock {\em Discrete and Computational Geometry}, 26(1):1--17, 2001.

\bibitem{MS15}
D.~Maclagan and B.~Sturmfels.
\newblock {\em Introduction to Tropical Geometry}, volume 161 of {\em Graduate
  Studies in Mathematics}.
\newblock American Mathematical Society, 2015.

\bibitem{Mor87}
A.~Morgan.
\newblock {\em Solving polynomial systems using continuation for engineering
  and scientific problems}, volume~57 of {\em Classics in Applied Mathematics}.
\newblock SIAM, 2009.

\bibitem{Pay09}
S.~Payne.
\newblock Fibers of tropicalization.
\newblock {\em Mathematische Zeitschrift}, 262:301--311, 2009.

\bibitem{PR12}
A.~Poteaux and M.~Rybowicz.
\newblock Good reduction of {P}uiseux series and applications.
\newblock {\em Journal of Symbolic Computation}, 47(1):32--63, 2012.

\bibitem{Sot11}
F.~Sottile.
\newblock {\em Real Solutions to Equations from Geometry}, volume~57 of {\em
  University Lecture Series}.
\newblock AMS, 2011.

\bibitem{ST08}
F.~Sottile and T.~Theobald.
\newblock Line problems in nonlinear computational geometry.
\newblock In J.E. Goodman, J.~Pach, and R.~Pollack, editors, {\em Computational
  Geometry - Twenty Years Later}, pages 411--432. AMS, 2008.

\bibitem{Sage}
W.\thinspace{}A. Stein et~al.
\newblock {\em {S}age {M}athematics {S}oftware ({V}ersion 6.9)}.
\newblock The Sage Development Team, 2015.
\newblock {\tt http://www.sagemath.org}.

\bibitem{Ver99}
J.~Verschelde.
\newblock Algorithm 795: {PHC}pack: A general-purpose solver for polynomial
  systems by homotopy continuation.
\newblock {\em ACM Trans. Math. Softw.}, 25(2):251--276, 1999.

\bibitem{Ver14}
J.~Verschelde.
\newblock Modernizing {PHC}pack through phcpy.
\newblock In P.~de~Buyl and N.~Varoquaux, editors, {\em Proceedings of the 6th
  European Conference on Python in Science (EuroSciPy 2013)}, pages 71--76,
  2014.

\end{thebibliography}
\end{document}